\documentclass[11pt,a4paper,reqno]{amsart}

\usepackage{mathtools}
\usepackage{dsfont}
\usepackage[shortlabels]{enumitem}
\usepackage[foot]{amsaddr}
\usepackage[margin=10pt,font=small,labelfont=bf]{caption}
\usepackage{graphicx}

\DeclareMathOperator{\ex}{\mathsf{E}}
\DeclareMathOperator{\prob}{\mathsf{P}}

\newcommand*{\abs}[1]{\left\lvert#1\right\rvert}
\newcommand*{\set}[1]{\left\{#1\right\}}
\newcommand{\FF}{\mathcal F}

\newcommand{\Real}{\mathbb R}
\newcommand{\ind}{\mathds 1}

\newtheorem{theorem}{Theorem}[section]
\newtheorem{proposition}[theorem]{Proposition}

\newtheorem{corollary}[theorem]{Corollary}
\theoremstyle{definition}

\theoremstyle{remark}
\newtheorem{remark}[theorem]{Remark}

\allowdisplaybreaks

\begin{document}

\title{Fractional diffusion Bessel processes with~Hurst~index~$H\in(0,\frac12)$}

\author{Yuliya Mishura{$^{1,2}$}}
\email{yuliyamishura@knu.ua}
\address{$^1$ Department of Probability Theory, Statistics and Actuarial Mathematics, Taras Shevchenko National University of Kyiv, 64/13, Volodymyrska Street, 01601 Kyiv, Ukraine}
\address{$^2$ Division of Mathematics and Physics, M\"alardalen University, 721 23 V\"aster{\aa}s, Sweden}
\author{Kostiantyn Ralchenko$^{1}$}
\email{kostiantynralchenko@knu.ua}

\begin{abstract}
We introduce fractional diffusion Bessel process with Hurst index $H\in(0,\frac12)$, derive a stochastic differential equation for it, and study the asymptotic properties of its sample paths.
\end{abstract}



\maketitle

\section{Introduction}

The study of diffusion Bessel processes began intensively with the paper \cite{McKean1} and was continued in many works, see, for example, \cite{PitmanYor81}, \cite[Ch.~11]{RevuzYor} and \cite{Yor97}.
For a positive integer $\nu$, a $\nu$-dimensional Bessel process $\rho$ can be defined as
\begin{equation}\label{eq:radialdef}
\rho_t = \sqrt{(B_t^1)^2+\dots+(B_t^\nu)^2},\quad t\ge0,
\end{equation}
where $B=(B^1,\dots,B^\nu)$ is a $\nu$-dimensional Brownian motion.
Note that $Z_t =\rho^2_t$ can be represented as a strong solution of the following stochastic differential equation:
\begin{equation}\label{eq:besq}
dZ_t = 2 \sqrt{\abs{Z_t}} dW_t + \nu dt,
\end{equation}
where $W$ is a one-dimensional Wiener process, that is related to $\nu$-dimen\-sio\-nal process $B$ by
$W_t = \sum_{i=1}^\nu \int_0^t \frac{B^i_s}{\rho_s}\,dB^i_s$.
This representation makes it possible to define Bessel process $\rho$ for every $\nu\ge0$ as a square root of $Z$, a unique solution to \eqref{eq:besq}, see, e.g., \cite[Ch.~11, Definition 1.9]{RevuzYor}.
Then for $\nu\ge2$, $\rho$ satisfies the equation
\begin{equation}\label{eq:difbes}
d\rho_t =  \frac{\nu-1}{2}\rho_t^{-1}\,dt+dW_t.
\end{equation}
For $\nu<2$ the situation is more involved due to the presence of local time, see e.g. \cite{bert, MPYT2023}.

There are several possibilities how to generalize the concept of Bessel processes to the fractional case.
For example, in \cite{ShevchenkoZatula} fractionally integrated Bessel process was studied. It is defined as
\[
\Upsilon_s=\int_0^t\left (u^{H-1/2} - (u - s)_+^{H-1/2}\right )d\rho_u,
\]
where $\rho$ is a Bessel process.
Fractional {R}iesz--{B}essel motion and related processes were studied in \cite{anh2001models}.
In the papers \cite{GuerraNualart,HuNualart2005}  (see also \cite{gaosome})
fractional Bessel process $\rho$ is defined by the formula \eqref{eq:radialdef}, where $\nu$-dimensional Brownian motion is replaced by $(B^{1,H},\dots,B^{\nu,H})$, a $\nu$-dimensional fractional Brownian motion with the Hurst index $H$.
For $\nu\ge2$ and $H>\frac12$ fractional Bessel process $\rho$ satisfies the equation
\begin{equation}\label{eq-nualart}
\rho_t = \rho_0 + H(\nu-1)\int_0^t\frac{s^{2H-1}}{\rho_s}\,ds+\sum_{i=1}^\nu \int_0^t \frac{B^{i,H}_s}{\rho_s}\,dB^{i,H}_s,
\end{equation}
where the stochastic integrals are interpreted in divergence form. The case $H<\frac12$ was studied in \cite{Essaky}.
Later this approach was extended to the cases of Bessel processes driven by bifractional \cite{SunGuoLi} and subfractional \cite{Shen-subfrac} Brownian motions.

Generally speaking, the process $\sum_{i=1}^\nu \int_0^t \frac{B^{i,H}_s}{\rho_s}\,dB^{i,H}_s$ from \eqref{eq-nualart} is not a fractional Brownian motion (except the case $H=\frac12$, where we get  Brownian motions).
This means that if we replace $W$ by a fractional Brownian motion in \eqref{eq:difbes}, we get another generalization of Bessel process to the fractional case. This generalization is the subject of the present paper.

\enlargethispage{4pt}More precisely, let $(\Omega,\FF,\prob)$ be a probability space and $B^H = \{B^H_t,t\ge\nolinebreak0\}$ be a fractional Brownian motion (fBm) with Hurst index $H\in(0,1)$, that is a Gaussian process with $\ex B^H_t =0$, $t\ge0$, and covariance function
$\ex B^H_tB^H_s = \frac12(t^{2H} + s^{2H} - \abs{t-s}^{2H})$,
$t,s \ge 0$.
The paper is devoted to the trajectories' behaviour of the process $X^H=\set{X^H_t, t\ge0}$ that satisfies the stochastic differential equation that is more convenient to write in the integral form:
\begin{equation}\label{eq:sde}
X_t^H = X_0 + a\int_0^t\frac{ds}{X_s^H} + \sigma B^H_t+L^H_t,\;t\ge 0,
\end{equation}
where $X_0$, $a$, $\sigma$ are strictly positive constants and the process $L^H$ is nondecreasing and appears for $H<1/2$. By analogy with diffusion Bessel processes, we name this process \emph{fractional diffusion Bessel process}. Let us denote $X_0$ the innitial value of all the equations under consideration.

Equation \eqref{eq:sde} is a particular case of equation
\begin{equation}\label{eq:sde-cir}
Y_t^H = X_0 + a\int_0^t\frac{ds}{Y_s^H}  - b\int_0^t Y_s^H ds + \sigma B^H_t+L^H_t,\;t\ge 0,
\end{equation}
where $b\ge0$.
Equation \eqref{eq:sde-cir} was investigated in detail in \cite{MSTA1} for $H>\frac12$ (in this case $L^H=0$). In \cite{MSTA2} an approach how to get this equation was established for $H<\frac12$, but equation itself was not considered.  See also \cite{MYT23}, where the case of small parameter $a>0$ was studied for $H>1/2$.
In particular, it was proved in \cite{MSTA1} that in the case $H>\frac12$ equation \eqref{eq:sde-cir} with $L^H=0$ has a unique solution that is strictly positive.
Moreover, according to comparison theorem for the equations involving fBm, that was proved in \cite[Lemma 1]{MSTA1}
the solution of equation \eqref{eq:sde} in the case $H>\frac12$ and with $L^H=0$ exists and is strictly positive.
The positivity of solution of \eqref{eq:sde} for $H>\frac12$ and with $L^H=0$ follows also from \cite{HuNualartSong} or \cite{Kubilius2020,KubMed}, where more general equations were studied. Moreover, \cite{Kubilius2020} investigates positive approximations for solutions and the parameter estimation problem for $H$.

However, for $H<\frac12$ the situation is much more involved.
We apply the ideas from \cite{MSTA2} in order to construct fractional diffusion Bessel process
by replacing the drift $\frac{a}{X_s}$ in \eqref{eq:sde} by
$\frac{a}{X_s \ind_{X_s>0} + \varepsilon}$; and by considering the pointwise limit of solutions as $\varepsilon\downarrow0$.
Then we derive a stochastic differential equation for this limit process and investigate its asymptotic behaviour for large time $t$. It turns out that the process does not hit zero starting with some random time, moreover, its asymptotic growth is comparable to a power function. We also study the influence of the parameter $a$ on the properties of the process. In particular, we show that its trajectory never hits zero for sufficiently large values of $a$. So,  in some sense, we finalize the preliminary information  about the fractional diffusion Bessel processes with $H<1/2$ obtained  in   \cite{MSTA2}.

The paper is organized as follows.
In Section \ref{sec:prelim} we define fractional diffusion Bessel process $X^H$ as a pointwise limit of solutions to stochastic differential equations, approximating the equation \eqref{eq:sde}. We also give some preliminary properties of $X^H$ related to its positivity and continuity.
In Section~\ref{sec:equation} we show that $X^H$ satisfies a stochastic differential equation of the form \eqref{eq:sde}, where the right-hand side contains  a reflection function $L^H$. We investigate the properties of this function and the asymptotic behaviour of $X^H$ as $t\to\infty$.
Section \ref{sec:funofdrift} is devoted to the asymptotic properties of a fractional diffusion Bessel process as a function of drift coefficient $a$; we study the behaviour of $X^H$ as $a\to\infty$ and  as $a\to0$.
Section~\ref{sec:numerics} contains some numerical illustrations for theoretical results established in Sections \ref{sec:equation}--\ref{sec:funofdrift}.

\section{Auxiliary properties of fractional Brownian motion and fractional diffusion Bessel process}
\label{sec:prelim}
The following statement from \cite{KMM2015} (see also \cite[Sec.~B.3.5]{KMR2017}) gives the a.s.\ upper bounds for fBm and its increments.
\begin{proposition} \label{prop:KMM-bounds}
\begin{enumerate}[(i)]
\item For any $p \ge 1$ the asymptotic growth of fBm is characterized as follows:
\[
\sup_{0< s \le t}
\abs{B^H_s} \le ((t^H \abs{\log t}^p) \vee 1) c(p),
\quad t > 0,
\]
where $c(p)=c(p,\omega)$ is a random variable having moments of any order.
\item For any $p \ge 1$ and $0 < \delta < H$ the asymptotic growth of the increments of
fBm is characterized as follows:
\[
\lvert B^H_s - B^H_t\rvert \le \lvert t - s\rvert^{H-\delta} ((t^\delta \abs{\log t}^p) \vee 1)c(p),
\quad t>s>0,
\]
where $c(p)=c(p,\omega)$ is a random variable having moments of any order.
\end{enumerate}
\end{proposition}

Since we are interested in the asymptotic results for large $t$, we can assume that $t\gg1$. In this case, we have an upper bound $\abs{\log t}  \le C t^\delta$ for any $\delta>0$.
This means that the results of Proposition \ref{prop:KMM-bounds} can be written in the following simplified form, which is more convenient for our calculations.

\begin{corollary}\label{boundsup}
\begin{enumerate}[(i)]
\item For any $\delta>0$ there exists $t_\delta(\omega)>0$ such that
\begin{equation}\label{eq:fbm-bound}
\sup_{0< s \le t}
\abs{B^H_s} \le c(\omega) t^{H+\delta},
\quad t > t_\delta(\omega),
\end{equation}
where $c(\omega)$ is a random variable having moments of any order.
\item For any $\delta>0$ and $0 < \beta < H$ there exists $t_{\beta,\delta}(\omega)>0$ such that
\begin{equation}\label{eq:incr-bound}
\abs{B^H_s - B^H_t} \le c(\omega)\abs{t - s}^{H-\beta} t^{\beta+\delta},
\quad t>s> t_{\beta,\delta}(\omega),
\end{equation}
where $c(\omega)$ is a random variable having moments of any order.
\end{enumerate}
\end{corollary}

Following \cite{MSTA2}, we can define the fractional diffusion Bessel process $X^H$ for $H < 1/2$ as the pointwise limit as $\varepsilon\downarrow0$ of the processes $Y^\varepsilon$ that satisfy the SDE of the
following type (also written in the integral form):
\begin{equation}\label{eq:sde-eps}
 X^\varepsilon_t =X_0+
a\int_0^t\frac{1}{X^\varepsilon_s \ind_{X^\varepsilon_s>0} + \varepsilon} ds + \sigma  B^H_t,
\quad
X_0 > 0.
\end{equation}
This limiting  process is well defined, as it is stated in the next proposition, and also has the following properties.
\begin{proposition}\label{prop:msta}
\begin{enumerate}[(i), left=-7pt]
\item The limit
$X^H_t(\omega) = \lim_{\varepsilon\to0} Y^\varepsilon_t(\omega)$
 exists, is nonnegative a.s.\ and is positive a.e.\ with
respect to the Lebesgue measure a.s.
\item $X^H$ satisfies the equation of the type
\[X_t^H = X_\alpha + a\int_\alpha^t\frac{ds}{X_s^H} + \sigma \left(B^H_t-B^H_\alpha\right),\]
for all $t \in [\alpha, \beta]$ where $(\alpha, \beta)$ is any interval of $X^H$'s strict positivity.

\item For all $t\ge 0$,
$\int_0^t\frac{1}{X_s^H}ds<\infty$.

\item For all $t \ge 0$,
$X_t^H \ge X_0 + a\int_0^t\frac{ds}{X_s^H} + \sigma B^H_t$.

\item The set $\{t > 0 \mid X^H_t > 0\}$ is open in the natural topology on $\Real$.

\item The trajectories of the process $X^H= \{X^H_t, t \ge0\}$ are continuous a.e.\ on $\mathbb R^+$ a.s.
\end{enumerate}
\end{proposition}

\section{Equation and asymptotic properties of the trajectories of fractional diffusion Bessel process with small Hurst index}
\label{sec:equation}

According to Proposition~\ref{prop:msta} $(iv)$, we have inequality
\begin{equation}\label{eq:sde-ineq}
X_t^H \ge X_0 + a\int_0^t\frac{ds}{X_s^H} + \sigma B^H_t.
\end{equation}
Denote
$L^H_t \coloneqq X_t^H - X_0 - a\int_0^t\frac{ds}{X_s^H} - \sigma B^H_t$.
Obviously we get the equation
\begin{equation}\label{eq:sde-lt}
X_t^H =  X_0 + a\int_0^t\frac{ds}{X_s^H} + \sigma B^H_t + L^H_t.
\end{equation}
The first result in this section establishes the properties of $L^H$.
\begin{theorem}
The process $L^H$ is nondecreasing, $L^H_0 = 0$, $L^H$ is continuous  a.e.\ on $\mathbb R^+$ a.s.\ and can increase only at points $t$ in which $X^H_t=0$.
In this sense, $L^H$ is a reflection function and the process $X^H$ satisfying the equation \eqref{eq:sde-lt} is a reflected fractional diffusion Bessel process.
\end{theorem}

\begin{proof}
Obviously, $L^H_0 = 0$, $L^H_t \ge 0$, and $L^H$ is continuous a.e. as a difference of continuous a.e.  process $X^H$ and continuous part  $X_0 + a\int_0^t\frac{ds}{X_s^H} + \sigma B^H_t$.
According to property $(ii)$ from Proposition~\ref{prop:msta}, at any point $t_0$, where $X^H_{t_0}>0$ (and which is surrounded by some interval $(t_0-h, t_0+h)$ where $X^H_{t_0}>0$) we have that
\[
X^H_{t_0+} - X^H_{t_0-} = a\int_{t_0-}^{t_0+}\frac{ds}{X_s^H}
+ \sigma\left(B^H_{t_0+} - B^H_{t_0-}\right),
\]
therefore, $L^H_{t_0+} - L^H_{t_0-} = 0$, and $L^H$ does not increase in $t_0$.
Finally, according to the Fatou lemma, similarly to the proof of Corollary 3.1 from \cite{MSTA2},
$\int_s^t \frac{du}{X^H_u} \le \liminf_{\varepsilon\to0} \int_s^t \frac{du}{Y^\varepsilon_u \ind_{Y^\varepsilon_u>0}+\varepsilon}$,
whence for any $0\le s<t$
\begin{align*}
X^H_t-X^H_s &= \lim_{\varepsilon\to0} \left(Y^\varepsilon_t - Y^\varepsilon_s\right)
= \int_s^t \frac{a du}{Y^\varepsilon_u \ind_{Y^\varepsilon_u>0}+\varepsilon}
+ \sigma \left(B^H_t-B^H_s\right)
\\
&\ge \int_s^t \frac{a du}{X^H_u}
+ \sigma \left(B^H_t-B^H_s\right),
\end{align*}
whence $L^H_t - L^H_s \ge 0$.
Theorem is proved.
\end{proof}

Consider now some asymptotic properties of trajectories of the fractional diffusion Bessel process with Hurst index $H<1/2$. In order not to repeat this every time, in all proofs we always assume that we consider only the arguments exceeding the values $t_{\delta}(\omega)$ or $t_{\beta,\delta}(\omega)$ from Corollary \ref{boundsup}. It is possible because in any proof this value will be fixed and will not change.  The first theorem states that  in the upper limit the trajectories exceed any power function with power index less than $1/2$ and is comparable to power function with power index   $1/2$. In some sense, this result is a bit unexpected because fractional Brownian motion itself is asymptotically negligible w.r.t. any power function with power index bigger than $H$, according to upper bound \eqref{eq:fbm-bound}. Without loss of generality and for the technical simplicity, let us assume that $\sigma=1$, but of course, all the results formulated below, are valid for fractional diffusion Bessel process with any $\sigma>0.$
\begin{theorem}\label{th:limsup}
\begin{enumerate}[(i),left=-14.65pt]
\item For any $\alpha\in(0,\frac12)$, with probability 1,
$\limsup\limits_{t\to \infty} \frac{X_t^H}{t^{\alpha}} = +\infty$.
\item With probability 1,
$\limsup\limits_{t\to \infty} \frac{X_t^H}{\sqrt{t}} \ge \sqrt{2a}$.
\end{enumerate}
\end{theorem}

\begin{proof}
$(i)$ As we know, $\prob(\Omega_0) = 1$,
$\Omega_0 = \set{\omega: c(\omega)<\infty}$, where $c(\omega)$ is taken from \eqref{eq:fbm-bound} Fix some $\omega\in \Omega_0$ and omit argument $\omega$ in what follows. In particular, denote $C_1=c(\omega)$.
We need to prove that for any $t\ge0$, $\alpha\in(0,\frac12)$, and $C>0$, there exists $t_1>t$ such that
$X^H_{t_1} \ge C t_1^\alpha$. Without loss of generality we can assume that $t>1.$
Assume the converse, i.e., there exist  $t>1$, $\alpha\in(0,\frac12)$, and $C>0$ such that
$X^H_{s} < C s^\alpha$ for any $s>t$.
Then for any $u>t$,
$a \int_t^u \frac{ds}{X^H_s} \ge \frac{a}{C} \cdot \frac{u^{1-\alpha} - t^{1-\alpha}}{1-\alpha}$.
According to upper bound \eqref{eq:fbm-bound},  $\abs{B^H_t}\le C_1 t^{H+\delta}$, where $\delta>0$ can be any number. In particular, we can chose $\delta>0$  sufficiently small, so that $H+\delta<1-\alpha$.
Then for any $u\ge t$
\begin{equation}\label{eq:contradiction}
Cu^\alpha > X^H_u \ge \frac{a}{C} \cdot \frac{u^{1-\alpha} - t^{1-\alpha}}{1-\alpha} -  C_1 u^{H+\delta}.
\end{equation}
It follows from \eqref{eq:contradiction} that for any $u\ge t$
\begin{equation}\label{eq:contradiction1}
 u^{1-\alpha} \ge t^{1-\alpha} +   \frac{C(1-\alpha) }{a} \cdot (Cu^\alpha+C_1 u^{H+\delta}),
\end{equation}  but $1-\alpha>\alpha\vee (H+\delta)$, and therefore inequality \eqref{eq:contradiction1} cannot hold for all $u>t$.
This contradiction proves $(i)$.

\bigskip

$(ii)$
We need to prove that for any $t\ge0$ and $\varepsilon<\sqrt{2a}$ there exists $t_1>t$ such that
$X^H_{t_1} > \varepsilon t_1^{1/2}$.
Assume that there exist $t\ge0$ and $\varepsilon<\sqrt{2a}$ such that $X^H_{u}\le \varepsilon u^{1/2}$ for all $u\ge t$.
As before, we can assume that $t>1$ and that $H+\delta<1/2$. Then for all $u\ge t$
\[
\varepsilon u^{1/2} \ge X^H_u \ge a \int_t^u\frac{ds}{X^H_s} -  C_1u^{H+\delta}
\ge \frac{2a}{\varepsilon}\left(u^{1/2} - t^{1/2}\right) -  C_1 u^{H+\delta},
\]
whence  for all $u\ge t$
\begin{equation}\label{eq:contradiction2}
\left(\frac{2a}{\varepsilon} - \varepsilon\right) u^{1/2}
\le \frac{2a}{\varepsilon} t^{1/2} + C_1 u^{H+\delta}.
\end{equation}
Due to inequality $\varepsilon<\sqrt{2a}$ we see that
$\frac{2a}{\varepsilon} - \varepsilon
>   0$, therefore
it is obvious that \eqref{eq:contradiction2} is impossible for all $u\ge t$.
\end{proof}
Now let's establish that the fractional diffusion Bessel process does not hit zero after some moment of time.

\begin{theorem}\label{th:liminf}
\begin{enumerate}[(i),left=0pt]
\item For any $H\in(0,\frac12)$ with probability 1,
$\liminf\limits_{t\to\infty} X^H_t  > 0$.

\item With probability 1 for any $\alpha>\frac12$,
$\limsup\limits_{t\to\infty} \frac{X^H_t}{t^\alpha} = 0$.
\end{enumerate}
\end{theorem}

\begin{proof}
$(i)$ Fix any $\omega_0\in\Omega_0$ and respective $C_1=c(\omega)$, where $\Omega_0$ was described in the proof of Theorem \ref{th:limsup}. Also, fix some sufficiently small $\delta>0$ and $\beta>0$ in \eqref{eq:incr-bound}, such that  $\beta+\delta<1, H+\beta+\delta<\frac12$ and consider $\alpha\in\left(H+\beta+\delta,\frac12\right)$. Finally, fix some   $t>1$.
According to Theorem \ref{th:limsup}, for any $C>C_1$ there exists $t_1\ge t$ depending on $C$ for which $X^H_{t_1}\ge   Ct_1^\alpha$.
Moreover, without loss of generality, we can assume that $t_1^{\alpha-H-\delta}>  2^{\delta+\beta}+1$ and $a^{1/2} t_1^{1/2-H-\delta}> \frac32C_1$.
Now our goal is to establish which lower level can be achieved by the process $X^H$ after moment $t_1$. Briefly speaking, our reasonings are of the following sort: assume that the process $X^H$ can achieve some small level. But for this time integral will increase, and it is necessary to compare its behaviour with the behaviour of fBm. In this connection, assume that there exists  such  $t_2>t_1$ that $\inf_{z \in [t_1,t_2]}X^H_{z}<\frac{C}{2}t_1^\alpha$ and  consider interval $[t_1,  t_2]$. Then
  there exists some $v\in[t_1,t_2)$ such that $X^H_z \ge   C z^{\alpha}$ for any $z\in[t_1,v ]$ but for any $z\in(v,t_2]$, $X^H_z < C v^{\alpha}$.
In this case  for any $z\in (v,t_2]$ we can bound $X^H_z$ from below, taking into account \eqref{eq:incr-bound}, considering the biggest possible value of the increment of fBm and taking it with minus, therefore considering the smallest possible value of $X^H_z$:
\begin{align}
 C v^{\alpha}&\ge X^H_z \ge X^H_v + a \int_v^z \frac{ds}{X^H_s} +  \left(B^H_z - B^H_v\right)
\notag\\
&\ge C v^{\alpha} + a\left(C v^{\alpha}\right)^{-1} (z-v) - C_1 z^{\beta+\delta} (z-v)^{H-\beta}.
\label{eq:lowbound}
\end{align}
Now the simplest way is to ignore the left-hand side of \eqref{eq:lowbound} considering instead  two cases.

$(a)$
Let $z-v \le v$.
Then obviously $z<2v$ and therefore
\begin{align*}
 X^H_z &\ge C v^\alpha - C_1 2^{\delta+\beta} v^{H+\delta}
\ge C_1 v^{H+\delta}\left( v^{\alpha-H-\delta} -   2^{\delta+\beta} \right)
\\
&\ge C_1 t_1^{H+\delta}\left(t_1^{\alpha-H-\delta} -   2^{\delta+\beta}\right)
\ge C_1 t_1^{H+\delta},
\end{align*}
according to assumption $t_1^{\alpha-H-\delta}>  2^{\delta+\beta}+1$.

$(b)$
Let $z-v > v$.
Then $z-v  >t_1$.  We can apply the inequality $b^2+d^2\ge 2bd$ to the sum $ C v^{\alpha} + a\left(C v^{\alpha}\right)^{-1} (z-v)$ and note that for small $\beta$ and $\delta$ such that $\beta+\delta<1$ we have that $z^{\beta+\delta}\le (z-v)^{\beta+\delta}+v^{\beta+\delta}\le 2(z-v)^{\beta+\delta}$.
Then we immediately get    that
\begin{align*}
X^H_z &\ge 2  a^{1/2} (z-v)^{1/2} - C_1 z^{\beta+\delta} (z-v)^{H-\beta}
\\
&\ge 2 (z-v)^{H+\delta} \left[a^{1/2} (z-v)^{1/2-H-\delta} - C_1
\right]
\\
&\ge 2v^{H+\delta} \left[  a^{1/2} v^{1/2-H-\delta} - C_1
\right]
\\
&\ge 2 t_1^{H+\delta}  \left[  a^{1/2} t_1^{1/2-H-\delta} - C_1
\right]
\ge C_1t_1^{H+\delta},
\end{align*}
according to assumption $a^{1/2} t_1^{1/2-H-\delta}> 3/2C_1$. It means that for any interval $[t_1, t_2],$ starting at the point $t_1$, we get the following alternative: either for any   $v\in[t_1,t_2)$ we have that  $X^H_v \ge   \frac{C}{2} t_1^{\alpha}$, or there exists some $v\in[t_1,t_2)$ such that $X^H_z \ge   C z^{\alpha}$ for any $z\in[t_1,v ]$, for any $z\in(v,t_2]$, $X^H_z < C v^{\alpha}$, however we established that in  this case  for any $z\in(v,t_2]$  $X^H_z  \ge C_1t_1^{H+\delta}$. In any case we get that $X^H_z>0$ for any $z\ge t_1$, and the statement $(i)$ is proved.

$(ii)$ Assume that there exists $t_n\to\infty$ such that
$\frac{X^H_{t_n}}{t_n^\alpha} = \varepsilon > 0$.
Consider
$\frac{X^H_{t_n}}{t_n^\alpha} =  \frac{X_0}{t_n^\alpha} + \frac{a}{t_n^\alpha} \int_0^{t_n}\frac{ds}{X_s^H} + \frac{B^H_{t_n}}{t_n^\alpha} + \frac{L^H_{t_n}}{t_n^\alpha}$,
and $\frac{X_0}{t_n^\alpha} \to 0$,
$\frac{B^H_{t_n}}{t_n^\alpha} \to 0$.
According to $(i)$, $X^H_t > 0$ starting from some $t_0$, therefore $L^H_t$ is constant starting from $t_0$, and
$\frac{L^H_{t_n}}{t_n^\alpha} \to 0$, $n \to \infty$.
Furthermore, according to Cauchy theorem, we can write
$\frac{a}{t_n^\alpha} \int_0^{t_n}\frac{ds}{X_s^H}
= \frac{a / X^H_{\theta_n}}{\alpha \theta_n^{\alpha-1}}
= \frac{a}{\alpha} \frac{\theta_n^{1-\alpha}}{X^H_{\theta_n}}$,
where we can assume without loss of generality that $\theta_n\to\infty$ as $n\to\infty$.
According to Theorem \ref{th:limsup}$(i)$,
$\frac{\theta_n^{1-\alpha}}{X^H_{\theta_n}} \to 0$, $n\to\infty$, and we get a contradiction, which proves $(ii)$.
\end{proof}

\section{Asymptotic properties of fractional diffusion Bessel process as a function of drift coefficient}
\label{sec:funofdrift}

Now  we consider the solutions of equation \eqref{eq:sde-lt} as the functions not only of $H$, but of $a$, and write them accordingly
\[
X_t^{H,a} = X_0 + a\int_0^t\frac{ds}{X_s^{H,a}} + \sigma B_t^H + L_t^{H,a}, \quad t\ge0.
\]
So,  $X^{H,a}$ is the notation for process $X^H$ if we wish to emphasize the dependence of its behavior of the value of $a>0$. First, let us consider the behavior of the trajectories of $X^{H,a} $ under increasing drift coefficient $a$. As before, without loss of generality, we assume that $\sigma=1.$
 
\begin{remark}
According to \cite[Thm.~3]{MSTA1},
$\prob\{\tau_0^{H,a}>T\} \to 1$,
as $a\to\infty$,
for any $T>0$, where
$\tau_0^{H,a} = \inf\{t>0 : X^{H,a}_t =\nolinebreak 0\}$.
\end{remark}

\begin{theorem}\label{th:large-a}
For any $\varepsilon>0$ there exists $a_0>0$ such that for any $a\ge a_0$
\[
\prob\Bigl\{\inf_{t\ge0} X^{H,a}_t = 0\Bigr\} < \varepsilon.
\]
\end{theorem}

\begin{proof}
Fix any $\varepsilon>0$ and consider $A>0$ such that
$\prob\set{c(\omega)>A} < \frac{\varepsilon}{2}$.
Consider
$\Omega_\varepsilon = \set{\omega : c(\omega)<A}$.
Then
$\prob\set{\Omega_\varepsilon} \ge 1-\frac{\varepsilon}{2}$.
Furthermore, in Theorem \ref{th:liminf} and its proof we can put
$\alpha = \frac{H}{2} + \frac14 \in (H,\frac12)$,
$a_0$ and $t_1$ such that
$t_1^{\alpha - H} = t_1^{\frac14 - \frac{H}{2}}
> 2^\delta + 1$,
$a_0^{\frac12} t_1^{\frac12 - H} >   A$.
Then on $\Omega_\varepsilon$
$X^{ H,a_0 }_u > 0$ for any $u\ge t_1$.
Obviously,
$X^{ H,a }_u > 0$ for any $u\ge t_1$ and $a\ge a_0$.
Now choose $a \ge a_0$ such that
$\prob\{\tau_0^{H,a} \le t_1\} < \frac{\varepsilon}{2}$.
Then
\[
\prob\set{\set{\tau_0^{H,a} > t_1} \cap \Omega_\varepsilon}
= 1 - \prob\set{\tau_0^{H,a}  \le t_1} -\prob\set{\overline{\Omega}_\varepsilon}
>1-\varepsilon,
\]
however, on the event
$\{\tau_0^{H,a} > t_1\} \cap \Omega_\varepsilon$
we have that $X^{H,a}_t > 0$ for any $t\ge0$.
\end{proof}

Now we consider the asymptotic behaviour of $X^{H,a}$ as $a\rightarrow 0.$
Let $a_1>a_2$. Then the solutions of the prelimit equations \eqref{eq:sde-eps}, considered with parameters $a_1$ and $a_2$, respectively, satisfy the conditions of comparison theorem \cite[Lemma 1]{MSTA1}, therefore,
$Y_t^{H,a_1,\varepsilon} \ge Y_t^{H,a_2,\varepsilon}$
a.s.,
where
$Y_t^{H,a_i,\varepsilon} = X_0 + a_i\int_0^t \frac{1}{Y_s^{H,a_i,\varepsilon}\ind_{Y_s^{H,a_s,\varepsilon}>0}+\varepsilon}\,ds +  \sigma B^H_t$,
$i=1,2$.
Taking limit in $\varepsilon\downarrow0$, we get that
$X_t^{H,a_1} \ge X_t^{H,a_2}$.
Therefore, for any $t\ge0$, $X_t^{H,a}$ decreases when $a$ decreases, and there exists a limit of $X_t^{H,a}$ as $a\downarrow0$:
$X_t^{H,0} \coloneqq \lim_{a\downarrow0} X_t^{H,a}$
a.s.
Obviously, $a\int_0^t\frac{ds}{X_s^{H,a}} + L_t^{H,a}$ also decreases in $a$, and we can introduce
$\widetilde L_t^H \coloneqq \lim_{a\downarrow0}(a\int_0^t\frac{ds}{X_s^{H,a}} + L_t^{H,a})$.
Then the process $X_t^{H,0}$ satisfies the equation
$X_t^{H,0} = X_0 + \sigma B^H_t + \widetilde L_t^H$,
$t\ge0$.
Obviously, process $\widetilde L_t^H$ increases only in the points where $X^{H,0}=0$, is nondecreasing and $X^{H,0}\ge0$.
In this sense, $\widetilde L_t^H$ is a version of reflection function for fractional Gaussian process $X_0+\sigma B^H$.

\section{Numerical illustrations}
\label{sec:numerics}

In this section we present several graphs illustrating our theoretical results.
We choose  $X_0 = 1$ and $H = 0.25$ for all our simulations and construct sample paths of a fractional diffusion Bessel process and related processes on the time interval $[0,1]$. In order to generate a solution of a stochastic differential equation, we consider a uniform partition of this interval with the step $10^{-6}$ and apply the Euler method.

Let us start with the case $a=1$.
Figure \ref{f:a1-lhs} contains a trajectory of $X^H$, constructed as a solution of \eqref{eq:sde-eps} with $\varepsilon = 10^{-4}$. Note that $X_H$ hits zero several times, but then, starting with some time, it never returns to zero.
This confirms the asymptotic property stated in Theorem \ref{th:liminf}.
Since the process $X^H$ can be constructed only at discrete points, it is problematic to illustrate continuity properties. However, the concentration of points around zeros may point at possible discontinuities.

\begin{figure}
    \centering
    \begin{minipage}{0.47\linewidth}
        \centering
        \includegraphics[width=\linewidth]{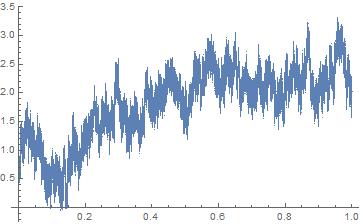}
\caption{Case $a=1$. A trajectory of $X^H$.}
\label{f:a1-lhs}
    \end{minipage}
    \begin{minipage}{0.47\linewidth}
        \centering
        \includegraphics[width=\linewidth]{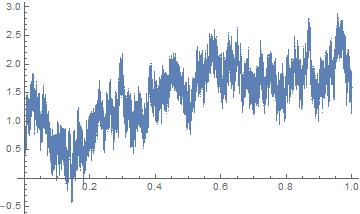}
\caption{Case $a=1$. The right-hand side of \eqref{eq:sde-ineq}, i.e., a trajectory of $X^H-L^H$}
\label{f:a1-rhs}
    \end{minipage}
\end{figure}

A trajectory of the right-hand side of \eqref{eq:sde-ineq} is displayed in Figure \ref{f:a1-rhs}.
We see that, unlike $X^H$, this process clearly has some points below zero. Moreover, the difference between left-hand and right-hand sides of \eqref{eq:sde-ineq}, i.e., the process $L^H$, is indeed positive as can be seen from its graph (Figure~\ref{f:a1-dif}).
This graph confirms also that $L^H$ is non-decreasing and increases only at the points $t$, where $X^H_t=0$. Moreover, $L^H$ is not continuous.

\begin{figure}
    \centering
    \begin{minipage}{0.47\linewidth}
        \centering
        \includegraphics[width=\linewidth]{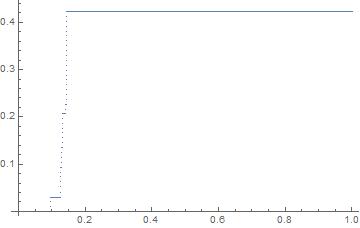}
\caption{Case $a=1$. A trajectory of $L^H$}
\label{f:a1-dif}
    \end{minipage}\hfill
    \begin{minipage}{0.47\linewidth}
        \centering
        \includegraphics[width=\linewidth]{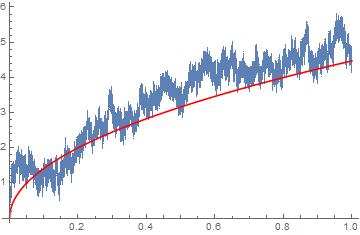}
\caption{Case $a=10$.  A trajectory of $X^H$ and a graph of $\sqrt{2at}$}
\label{f:a10-lhs}
    \end{minipage}
\end{figure}

In order to illustrate the behavior of the fractional diffusion Bessel process for large values of $a$, let us take $a=10$. A corresponding trajectory of $X^H$, constructed as a solution of \eqref{eq:sde-eps} with $\varepsilon = 10^{-4}$, is represented in Figure~\ref{f:a10-lhs}. We see that in this case the process $X^H$ never hits zero; this, in some sense,  agrees with Theorem~\ref{th:large-a}. Moreover, its asymptotic growth is comparable with a function $\sqrt{2a t}$ (red curve) as stated in Theorem \ref{th:limsup}.
\begin{figure}[b]
\centering
    \begin{minipage}{0.47\linewidth}
        \centering
        \includegraphics[width=\linewidth]{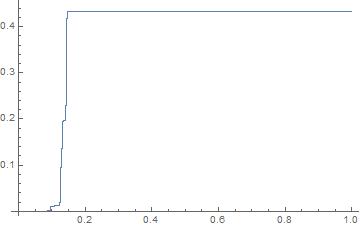}
\caption{Case $a=0.01$. A trajectory of $\widetilde{L}^H$.}
\label{f:a001-dif}
    \end{minipage}
\end{figure}

Finally, in order to illustrate the behaviour of $\widetilde L^H$  let us take a small value of $a$, namely $a=0.01$.
Figure \ref{f:a001-dif} contains a corresponding trajectory of $\widetilde L^H$. We see that it is still non-negative and non-decreasing, which confirms our theoretical results. Moreover, we see that $\widetilde L^H$ has more points of growth, in comparison  to $L^H$ which testifies in favor of the fact that the reflecting function $\widetilde L^H$ is the limit of two components:  the prelimit integral and the prelimit reflecting function.

\section*{Acknowledgements}
The first author  was supported by The Swedish Foundation for Strategic Research, grant Nr.\ UKR22-0017 and by Japan Science and Technology Agency CREST JPMJCR2115.
The second author is grateful to his hosts at Macquarie University, where he was a Visiting Fellow
sponsored by the Sydney Mathematical Research Institute under Ukrainian Visitors Program (UVP22).
The authors acknowledge that the present research is carried through within the frame and support of the ToppForsk project nr. 274410 of the Research Council of Norway with title STORM: Stochastics for Time-Space Risk Models.

\bibliographystyle{abbrv}
\bibliography{fbessel}
\end{document}